\documentclass[reqno]{amsart}
\usepackage{mathrsfs}
\usepackage{graphicx}
\usepackage{amscd}
\usepackage{amsmath}
\usepackage{amsthm}
\usepackage{amsfonts}
\usepackage{amssymb}
\usepackage{amsxtra}
\usepackage{xspace}
\usepackage{array}
\usepackage{cite}
\usepackage{color}
\usepackage{xypic}

\theoremstyle{definition}

\newtheorem{theorem}{Theorem}[section]
\newtheorem{corollary}[theorem]{Corollary}
\newtheorem{lemma}[theorem]{Lemma}
\newtheorem{proposition}[theorem]{Proposition}

\newtheorem{example}{Example}

\newcommand{\kma}{\mathfrak{g}}
\newcommand{\csa}{\mathfrak{h}}

\newcommand{\mcg}{G}

\newcommand{\form}[2]{(#1 \mid #2)}

\newcommand{\rop}{\Delta_+}

\newcommand{\sr}{\Pi}
\newcommand{\integers}{\mathbb{Z}}

\newcommand{\pp}{\mathcal{P}}
\newcommand{\bb}{\mathcal{B}}

\newcommand{\be}{\begin{enumerate}}
\newcommand{\ee}{\end{enumerate}}
\newcommand{\beq}{\begin{equation}}
\newcommand{\eeq}{\end{equation}}

\newcommand{\pg}[1][q]{\chi(G; \,#1)}
\newcommand{\pgsig}[2][q]{\chi(G(#2); \,#1)}
\newcommand{\ptg}[1][q]{\widetilde{\chi}(G; \,#1)}

\newcommand{\bl}[1][G]{L_{#1}}
\newcommand{\kost}[1][\beta]{K(#1;\,q)}
\newcommand{\mob}{\mu}
\newcommand{\zhat}{\hat{0}}
\newcommand{\covers}{\rightarrow}
\newcommand{\ed}{\boldsymbol{e}}

\newcommand{\nminus}[1][-]{\mathfrak{n}^{#1}}

\DeclareMathOperator{\mult}{mult}

\DeclareMathOperator{\start}{start}
\DeclareMathOperator{\eend}{end}
\DeclareMathOperator{\len}{len}
\DeclareMathOperator{\supp}{supp}
\DeclareMathOperator{\spann}{span}

\begin{document}

\title[]{Chromatic polynomials of graphs from Kac-Moody algebras}
\author{R. Venkatesh}
\author{Sankaran Viswanath}
\address{Tata Institute of Fundamental Research, Homi Bhabha Road, Colaba, Mumbai 400005, India.}
\email{rvenkat@math.tifr.res.in}
\address{The Institute of Mathematical Sciences, CIT campus, Taramani,
Chennai 600113, India.}
\email{svis@imsc.res.in}
\begin{abstract}
We give a new interpretation of the chromatic polynomial of a simple graph $G$ in terms of the Kac-Moody Lie algebra $\kma$ with Dynkin diagram $G$. We show that the chromatic polynomial is essentially the $q$-Kostant partition function of $\kma$ evaluated on the sum of the simple roots. 
Applying the Peterson recurrence formula for root multiplicities of $\kma$, we obtain a new realization of the chromatic polynomial as a weighted sum of paths in the bond lattice of $G$.
\end{abstract}
\maketitle

\section{Introduction}

Let $\mcg$ be a simple graph on $l$ vertices. Let $\kma$ denote the
Kac-Moody Lie algebra with Dynkin diagram $\mcg$. In other words,
first define the $l \times l$ symmetric matrix $B:= 2I - A$ where $A$ is the adjacency matrix of $\mcg$. Then $B$ is a generalized Cartan matrix, and $\kma$ is the Kac-Moody algebra constructed from $B$. 
 We do not assume that $G$ is connected; so the Lie algebra $\kma$ is a direct sum of the Kac-Moody algebras corresponding to the connected components of $G$. 
In this paper, we obtain a new expression for the chromatic polynomial of $\mcg$ in terms of the Lie algebra $\kma$.

\subsection{}
We recall that the Kac-Moody algebra $\kma$ with Dynkin
diagram $G$ is a complex Lie algebra defined via
generators-and-relations \cite{kac}. The Lie algebra $\kma$ is
finite-dimensional precisely when every connected component of $G$ is
a Dynkin diagram of type $A_n \, (n \geq 1)$, $D_n\, (n \geq 4)$ or $E_n\, (n=6, 7, 8)$.

We let the root space decomposition of $\kma$ be 
 $$\kma = \csa \oplus \bigoplus_{\alpha \in \rop} (\kma_\alpha \oplus \kma_{-\alpha})$$ where $\csa$ is the Cartan subalgebra, $\rop$ is the set of positive roots, and $\kma_\alpha:=\{x \in \kma: [h,x]=\alpha(h) \, x \, \text{ for all } h \in \csa\}$ are the root spaces. Let $\mult \alpha := \dim \kma_{\alpha}$ be the multiplicity of the root $\alpha$. We also let $\nminus[\pm] := \bigoplus_{\alpha \in \rop} \kma_{\pm \alpha}$.

Let $\sr$ denote the set of simple roots of $\kma$. 
We identify $\sr$ with the vertex set of $G$.
Given a subset $S \subset \sr$, define $G(S)$ to be the subgraph of
$G$ induced by $S$, and let
$$\beta(S) := \sum_{\alpha \in S} \alpha.$$
It is well-known that $\beta(S)$ is a root of $\kma$ (i.e., $\mult
\beta(S) > 0$) iff  $G(S)$ is connected. Further, $\mult \beta(S)$
depends only on $G(S)$ and not on the ambient graph $G$ \cite{kac}. 

Now, let $\pi = \{S_1, S_2, \cdots, S_k\}$ be a partition of $\sr$, i.e., the $S_i$ are non-empty pairwise disjoint subsets of $\sr$ such that $\bigcup_i S_i = \sr$. Define  
$$\mult \pi := \prod_{i=1}^k \mult \beta(S_i).$$
Thus $\mult \pi > 0$ iff each $S_i$ induces a connected subgraph of $G$, or in other words, iff $\pi$ is an element of the {\em bond lattice} of $G$. We recall \cite{rota,stanley-sfgcp} that the bond lattice $\bl$ of $G$ is the  set of all partitions $\pi = \{S_1, S_2, \cdots, S_k\}$ of $\sr$ such that each $S_i$ induces a connected subgraph of $G$ 
(such $S_i$ will be called {\em connected subsets} of $\sr$). It is partially ordered by refinement, with $\pi > \pi^\prime$ iff $\pi^\prime$ refines $\pi$. 
$\bl$ is a ranked poset, with rank function $l - |\pi|$, where $|\pi|$ is the number of parts in $\pi$ ($=k$ if $\pi = \{S_1, S_2, \cdots, S_k\}$). The partition $\zhat$ of $\sr$ into $l$ singleton subsets is the unique minimal element
of  $\bl$.

\subsection{}
We now recall some standard notions about chromatic polynomials. For each positive integer $q$, let $\pg$ be the number of {\em proper colorings} of $G$ with $q$ colors, i.e., the number of ways of assigning one of $q$ colors to each vertex of $G$ such that no two adjacent vertices have the same color.
Then $\pg$ is a polynomial in $q$, and is called the chromatic polynomial of $G$.

We will identify the set $\sr$ of simple roots of $\kma$ with the vertices of
$G$ as above. A non-empty subset $K \subset \sr$ is said to be {\em independent} if no two vertices in $K$ have an edge between them. For $k \geq 1$, let $P_k(\mcg)$ denote the set of ordered partitions of $\sr$ into $k$ independent sets, i.e., $P_k(\mcg)$ is the set of ordered $k$-tuples $(J_1,...,J_k)$ such that 
(a) the $J_i$'s are non-empty pairwise disjoint subsets of $\sr$, 
(b) $\bigcup\limits_{i=1}^k J_i=\sr$, and (c) each $J_i$ is independent.
Let $c_k(\mcg) :=|P_k(\mcg)|$. Then, the chromatic polynomial of $\mcg$ has the following well-known expression (cf., \cite{rcread}):
\beq\label{cpe}
\pg = \sum_{k \geq 1} c_k(\mcg) \, {q \choose k}.
\eeq

\subsection{}

For $\Sigma=\{S_1, \cdots, S_k\} \in \bl$, we let $G(\Sigma)$ denote the union of the subgraphs of $G$ 
induced by $S_i$, $i=1\cdots k$. Our first main theorem 
relates the chromatic polynomial of $G$ to root multiplicities of $\kma$.

\begin{theorem} \label{pkost}
Let $G$ be a simple graph. With notation as above, we have
\be
\item[(a)]
\beq\label{pkosteq}
\pg = \sum_{\pi \in \bl} (-1)^{l - |\pi|} \, \mult \pi \;\, q^{|\pi|}.
\eeq
\item[(b)] 
  More generally, let $\Sigma \in \bl$. Then 
\beq\label{pkosteq2}
\pgsig{\Sigma} = \sum_{\zhat \leq \pi \leq \Sigma} (-1)^{l - |\pi|} \, \mult \pi \;\, q^{|\pi|}.
\eeq
\ee
\end{theorem}

We consider two simple examples to demonstrate this theorem.
\begin{example}
Let $G$ be a tree on $l$ vertices. In this case, it is easy to see
that $\beta(S)$ is a {\em real
  root} of $\kma$ for each connected
subset $S \subset \sr$. In particular this means that $\beta(S)$ has
multiplicity 1 and hence $\mult \pi = 1$ for all $\pi \in \bl$. Now, the
number of $\pi \in \bl$ with $|\pi| = k$ is ${l-1} \choose {k-1}$ for
$1 \leq k \leq l$.  Plugging this into equation \eqref{pkosteq}, we
recover the well-known expression: $\pg = q(q-1)^{l-1}$.
\end{example} 

\begin{example}
Let $G$ denote a cycle graph with $l$ vertices, where $l \geq 3$. This is the Dynkin
diagram of the affine Kac-Moody algebra $A_{l-1}^{(1)}$
\cite{kac}. Let $S$ be a {\em proper} connected subset of $\pi$. Then, as above, $\beta(S)$ is a real root of $\kma$ and has
multiplicity 1. However, if $S = \sr$, then $\beta(S)$ is the
so-called null root of this affine Kac-Moody algebra, and has
multiplicity $l-1$ \cite[Chapter 7]{kac}. The number of $\pi \in \bl$
with $|\pi| = k $ is now $l \choose k$, for $2 \leq k \leq l$, and is
1 if $k=1$. Putting all this into equation \eqref{pkosteq}, one
obtains the well-known expression
$\pg = (q-1)^l + (-1)^l (q-1)$. 
\end{example}

\subsection{}
We recall that $\pg$ is monic of degree $l$ with coefficients that alternate in sign, and hence $\ptg := (-1)^l \,\pg[-q]$ has non-negative coefficients. The following is a pleasant consequence of Theorem \ref{pkost}.
\begin{corollary} \label{kostcor}
Let $G$ be a simple graph. Then
$\ptg = \kost[\,\beta(\sr)]$,  where $K(\, \cdot \, ; q)$ is the $q$-Kostant partition function of $\kma$.
 \end{corollary}

The $q$-Kostant partition function $\kost$ is defined to be the coefficient of $e^{\beta}$ in the product 
\beq \label{qkosdef}
\prod_{\alpha \in \rop} \left( 1 - qe^{\alpha}\right)^{-\mult \alpha}.
\eeq
The universal enveloping algebra $U\nminus[+]$ has a natural
increasing filtration by the subspaces
$$\left(U\nminus[+]\right)^{\, \leqslant n} := \spann \{x_1x_2\cdots
x_k: \;  k \leqslant n \text{ and } x_j \in \nminus[+] \;
\forall j\}$$
for $n=0, 1, 2, \cdots$. Let us also set $
\left(U\nminus[+]\right)^{\, \leqslant n} = 0 $ for $n<0$. 
This so-called {\em degree filtration} induces a filtration on each
weight space of $U\nminus[+]$. A straightforward application of the Poincar\'{e}-Birkhoff-Witt
theorem shows that $\kost$ is precisely the
Hilbert series of the associated graded space of the degree filtration on 
$(U\nminus[+])_\beta$, the $\beta$-weight space of $U\nminus[+]$, i.e., 
$$\kost = \sum_{n \geq 0} \dim \left( \frac{(U\nminus[+])_\beta \cap \left(U\nminus[+]\right)^{\, \leqslant n} }{(U\nminus[+])_\beta \cap \left(U\nminus[+]\right)^{\, \leqslant n-1} } \right) \; q^n.$$

\medskip
We use this interpretation to obtain a further corollary of Theorem \ref{pkost} by considering special values of $q$. First, recall that $\ptg[1]$ is the number of acyclic orientations of $G$ \cite{stanley-acyc}, or equivalently the number of distinct Coxeter elements in the Weyl group $W(\kma)$ \cite{jyshi1}. Similarly, if $G$ is connected, then the coefficient of $q$ in $\ptg$ is the number of conjugacy classes of Coxeter elements in $W(\kma)$ \cite{jyshi2,erik2,mac-mor}. 

\begin{corollary} \label{kostliecorr}
Let $\beta :=\beta(\sr)$.
\be
\item The number of acylic orientations of $G$ equals $K(\beta)$, where  $K(\cdot)$ is the Kostant partition function of $\kma$. In other words, the number of distinct Coxeter elements in $W(\kma)$ equals $\dim \,(U\nminus[+])_\beta$.
\item  If $G$ is connected, then the number of conjugacy classes of Coxeter elements in $W(\kma)$ equals $\dim \, \nminus[+]_\beta$ ($= \mult \beta$).
\ee
\end{corollary}

Theorem \ref{pkost} and its corollaries are proved in section \ref{two}.
The key ingredient in the proof is Lemma \ref{keylemma} below, which
is a special case of a result proved in \cite{rvsv} in the context of
unique factorization of tensor products for Kac-Moody algebras. 
In fact, it was the occurrence of the {\em deletion-contraction}
recurrence in \cite{rvsv} that suggested a possible connection of those
ideas to chromatic polynomials.

\subsection{}
We note that equation \eqref{pkosteq} closely resembles 
the classical result of Birkhoff and Whitney which essentially states that
\beq\label{birwhit}
\pg = \sum_{\pi \in \bl} \mob(\zhat,\pi) \, q^{|\pi|}
\eeq
where $\mob$ is the  M\"{o}bius function of $\bl$. 
The following proposition (proved in section \ref{two}) clarifies the relation between the two.

\begin{proposition} \label{mobmult}
For all $\pi \in \bl$, $\mob(\zhat,\pi) =  (-1)^{l - |\pi|} \,\mult \pi$.
\end{proposition}
Thus, the absolute value of the M\"{o}bius function is a product of
certain root multiplicities of $\kma$. This provides an interesting Lie algebraic interpretation of the  M\"{o}bius function.

\subsection{}
Now, root multiplicities are themselves quite mysterious in general (except when $\kma$ is of finite or affine type), and it is natural to wonder if this  interpretation sheds any further light on the chromatic polynomial. However, an important property of root multiplicities is that they satisfy the so-called {\em Peterson recurrence} (see section \ref{three} below). This, together with Theorem \ref{pkost}, allows us to find a new realization of the chromatic polynomial, as a weighted generating function of paths in the bond lattice. 

In order to describe this realization more precisely, we require some definitions. 
 For a subset $S$ of $\sr$, let $\ed(S)$ denote the number of edges in the subgraph induced by $S$ (i.e., the number of edges of $\mcg$ both of whose ends are in $S$). 
 Given $\pi \in \bl$, let $d(\pi)$ denote the number of non-singleton subsets in the partition $\pi$, i.e., if $\pi = \{S_1, S_2, \cdots, S_k\}$, then 
$$ d(\pi) := \# \{1 \leq i \leq k: |S_i| > 1\}.$$
Observe that $d(\pi) =0 $ iff $\pi = \zhat$.

The bond lattice $\bl$ can be thought of as a directed graph, with directed edges given by the covering relations, i.e., given $\pi, \pi^\prime \in \bl$, 
we draw an edge from $\pi$ to $\pi^\prime$ iff $\pi \covers \pi^\prime$. 
We observe that if $\pi \covers \pi^\prime$, then $|\pi^\prime| = |\pi|+1$; further, we can write 
$\pi=\{S_1, S_2, \cdots, S_k\}$ and $\pi^\prime =\{S^\prime_1, S^\prime_2, \cdots, S^\prime_{k+1}\}$ with $S_i = S^\prime_i$ for $1 \leq i < k$ and $ S_k =  S^\prime_{k} \sqcup S^\prime_{k+1}$. Define a (rational valued) weight function on the edge $\pi \covers \pi^\prime$ of $\bl$ as follows:
$$ 
w(\pi, \pi^\prime):= 
\frac{1}{d(\pi)} \frac{\ed(S^\prime_{k},S^\prime_{k+1})}{\ed(S_k)}
=\frac{1}{d(\pi)} \left( 1 - \frac{\ed(S^\prime_{k})}{\ed(S_k)} - \frac{\ed(S^\prime_{k+1})}{\ed(S_k)}\right)
$$
 where $\ed(S^\prime_{k},S^\prime_{k+1})$ is the number of edges of $\mcg$ which straddle $S^\prime_{k}$ and $S^\prime_{k+1}$, i.e., one end of which lies in $S^\prime_{k}$ and the other in $S^\prime_{k+1}$.
Observe that $\pi$ covers an element of $\bl$ implies that $\pi \neq \zhat$, and hence $d(\pi)\neq 0$. Since $S_{k}$, $S^\prime_{k}$ and $S^\prime_{k+1}$ induce connected subgraphs of $G$, it is clear that $0 < w(\pi, \pi^\prime) \leq 1$, and 
$w(\pi, \pi^\prime) = 1$ iff $\pi^\prime =\zhat$.

Now, given a (directed) path $p$ in $\bl$, say $p: \pi_1 \covers \pi_2 \covers \cdots \covers \pi_r$, we let $\start(p) :=\pi_1$, $\eend(p) := \pi_r$ and $\len(p) := r-1$. Define the weight of $p$ to be the product of the weights of edges in $p$, i.e., 
$$ w(p) := \prod_{i=1}^{r-1} w(\pi_i,\pi_{i+1}).$$ 
If $r=1$, i.e., $p$ is a path of length zero, then this is an empty product, and $w(p):=1$. We now  have the following proposition, which arises from an
 iterated application of the Peterson recurrence formula. The proof appears in section \ref{three}.

\begin{proposition}\label{mpipath}
Let $\pi \in \bl$. Then 
\beq \label{mpi}
\mult \pi = \sum_{\substack{p \text{ path in } \bl \\ \start(p) = \pi \\ \eend(p) = \zhat}} w(p).
\eeq
\end{proposition}
The terms appearing on the right hand side of equation \eqref{mpi} are all rationals between 0 and 1, and it seems somewhat remarkable that their sum finally works out to be an integer for every $\pi$.
If $p$ is a path from $\pi$ to $\zhat$, then observe $|\pi| = l - \len(p)$.
Thus, Theorem \ref{pkost} and proposition \ref{mpipath} imply the following new realization of the chromatic polynomial.
\begin{theorem}\label{mainthm}
$$\pg = \sum_{\substack{p \text{ path in } \bl \\ \eend(p) = \zhat}} (-1)^{\len(p)} \, w(p) \, q^{l - \len(p)}.$$
\end{theorem} 

Next, we define a square matrix $Z$ of order $L:=|\bl|$ (with rows and columns indexed by elements of $\bl$) as follows:
$$ Z_{\pi, \pi^\prime} :=
\begin{cases} 
- w(\pi,\pi^\prime) & \text{if } \pi \covers \pi^\prime \\
\;\;\;\;\; q & \text{if } \pi = \pi^\prime = \zhat \\
\;\;\;\;\; 0 & \text{otherwise}.
\end{cases}$$
It is clear that $Z$ is an upper triangular matrix (after rearranging the elements of $\bl$ in decreasing order with respect to a linear extension) and has diagonal entries $0,0,\cdots, 0$ ($L-1$ times) and $q$.

Let $\zeta$ denote the column vector of length $L$ all of whose entries are 1. Then it is easy to see that 
the following is an alternative formulation of Theorem \ref{mainthm}.
\begin{corollary} \label{mainthmcor}
$\pg = \zeta^{\scriptscriptstyle T} \, Z^l \, \zeta$.
\end{corollary}

The rest of the paper is organized as follows. Section \ref{two} contains preliminaries about Kac-Moody algebras, and the proofs of Theorem \ref{pkost}, proposition \ref{mobmult} and corollaries \ref{kostcor}, \ref{kostliecorr}. In section \ref{three}, we describe the Peterson recurrence formula, and use it to prove proposition \ref{mpipath}.

\section{Proof of Theorem \ref{pkost}}\label{two}

\subsection{} 
We freely use the notations of the introduction in the rest of the paper. Thus, $\kma$ will denote the Kac-Moody algebra with Dynkin diagram $\mcg$. Let $\csa$ be its Cartan subalgebra and $\rop$ the set of positive roots. The simple roots of $\kma$ will be identified with $\Pi$. The Weyl group $W$ of $\kma$ is the subgroup of $\mathrm{GL}(\csa^*)$ generated by the simple reflections $\{s_\alpha: \alpha \in \sr\}$. We let $Q_+ := \oplus_{\alpha \in \Pi} \integers_{\geq 0} \alpha$ be the positive part of the root lattice. We also have a natural non-degenerate symmetric bilinear form on $\csa^*$ \cite{kac}. On the simple roots, it is given by $\form{\alpha}{\beta} = 2$ if $\alpha = \beta$, $-1$ if $\alpha$ and $\beta$ are adjacent vertices in $\mcg$, and $0$ otherwise.
We let $\rho \in \csa^*$ denote the Weyl vector of $\kma$; this satisfies $\form{\rho}{\alpha} = 1$ for all $\alpha \in \Pi$.

\subsection{} Let $\pp$ denote the collection of sets $\{\beta_1,\cdots, \beta_k\}$ such that (i) each $\beta_i \in \rop$ and 
(ii) $\sum \beta_i = \beta(\sr)$. We partially order $\pp$ by refinement, i.e., given $\gamma = \{\beta_1,\cdots, \beta_k\}$ and  $\gamma' = \{\beta'_1,\cdots, \beta'_r\}$ in $\pp$, define $\gamma > \gamma'$ if there exist pairwise
disjoint sets $K_i \, (1 \leq i \leq k)$ such that $\bigcup K_i = \{1,\cdots,r\}$ and $\beta_i = \sum_{j \in K_i} \beta'_j$ for all $1 \leq i \leq k$. The covering relation in $\pp$ is thus obtained by $\gamma \covers \gamma'$ iff $r=k+1$ and (after possibly reordering indices) $\beta_i = \beta'_i$ for $1 \leq i < k$ and $\beta_k = \beta'_k + \beta'_{k+1}$.

We make the following simple observation.
\begin{lemma}\label{posiso}
The map $\phi: \bl \to \pp$ defined by $\{S_1, \cdots, S_k\} \mapsto \{\beta(S_1), \cdots, \beta(S_k)\}$ is an isomorphism of posets.
\end{lemma}
\begin{proof}
We recall that $S$ is a connected subset of $\sr$ iff $\beta(S) \in \rop$. Thus $\phi$ is well defined. The fact that it is an isomorphism is clear.
\end{proof}
We can thus identify $\bl$ and $\pp$. We will let $\zhat$ also denote the unique minimal element of $\pp$ (the set of all simple roots of $\kma$). For $\gamma =\{\beta_1, \cdots, \beta_k\} \in \pp$, we  let $\mult \gamma := \prod_{i=1}^k \mult \beta_i$.

The Weyl-Kac denominator formula gives:
\beq\label{wk}
 U_0 := \sum\limits_{w \in W} \varepsilon(w) e^{w\rho - \rho} = \prod\limits_{\alpha \in \rop} (1 - e^{-\alpha})^{\mult \alpha},
\eeq
where $\varepsilon$ is the sign character of $W$.
\begin{proposition} \label{lemco}
The coefficient of $e^{-\beta(\sr)}$ in $U_0^q$ equals $(-1)^l \, \pg$.
\end{proposition}
\begin{proof}
We have $U_0^q =\sum\limits_{k\geq 0} {q \choose k}\,\xi^k$ where $\xi:=U_0 - 1 =  
\sum\limits_{w \neq e} \varepsilon(w) \, e^{w\rho - \rho}$. From equation \eqref{cpe}, it is clear that we only need to show that the coefficient of 
$e^{-\beta(\sr)}$ in $\xi^k$ equals $(-1)^l \,c_k(\mcg)$. 

For $w \in W$, let $\rho-w\rho=\sum\limits_{\alpha\in \sr}b_\alpha(w)\,\alpha$; we have  $b_\alpha(w) \in \integers_{\geq 0}$.  We also define 
$I(w):=\{\alpha\in \sr: s_\alpha \text{ appears in a given reduced word for } w\}$; 
this is independent of the reduced word chosen \cite{humphreys}. Let $\mathcal{I}:=\{w\in W\backslash \{e\}: I(w) \text{ is an independent set}\}$. The following lemma is a special case of lemma 2 of \cite{rvsv}. 

\begin{lemma}\label{keylemma}
Let $w\in W$. Then
\be
\item[(a)] $I(w)=\{\alpha \in \sr: b_\alpha(w) \geq 1\}$.
\item[(b)] If $w \in \mathcal{I}$, then $b_\alpha(w)= 1$ for all $\alpha \in I(w)$.
\item[(c)] If $w\notin \mathcal{I} \cup \{e\}$, then there exists $\alpha \in I(w)$ such that $b_{\alpha}(w)> 1$.
\ee
\end{lemma}
Given an independent subset $K$ of $\sr$, there is a unique element $w(K) \in \mathcal{I}$ with $I(w(K)) = K$; $w(K)$ is the product of the commuting simple reflections $\{s_\alpha: \alpha \in K\}$.  Now, it follows from lemma \ref{keylemma} that the coefficient of $e^{-\beta(\sr)}$ in $\xi^k$ equals 
$$\sum_{(w_1,\cdots,w_k)} \varepsilon(w_1w_2\cdots w_k)$$
where the sum ranges over $k$ tuples $(w_1,\cdots,w_k)$ such that $w_i \in \mathcal{I}$ for all $i=1,\cdots,k$ and  $(I(w_1), \cdots, I(w_k)) \in P_k(\mcg)$. In this case, $w_1w_2\cdots w_k$ is a Coxeter element of $W$, and has sign $(-1)^l$. Thus, the required coefficient is $(-1)^l \,c_k(\mcg)$, and proposition \ref{lemco} is proved. 
\end{proof}

\subsection{}
Now, using the product side of the Weyl-Kac denominator formula (equation \eqref{wk}), we obtain 
\beq\label{prodside}
U_0^q = \prod\limits_{\alpha \in \rop} (1 - e^{-\alpha})^{q \mult \alpha} =  \prod\limits_{\alpha \in \rop} (1 -  e^{-\alpha} \, q \mult \alpha + {O}(q^2)).
\eeq
Observe that the terms involving $q^2$ (and higher powers) also involve $e^{-n \alpha}$ for $n \geq 2$, and do not contribute to the coefficient of $e^{-\beta(\sr)}$. Thus the coefficient of $e^{-\beta(\sr)}$ in $U_0^q$ is 
$\sum_{\gamma \in \pp} (-q)^{|\gamma|} \, \mult \gamma$.
Lemma \ref{posiso} and proposition \ref{lemco} now complete the proof of Theorem \ref{pkost}{(a)}. 

To prove Theorem \ref{pkost}{(b)}, observe that for $\Sigma \in \bl$, we can identify the poset $\bl[G(\Sigma)]$ with the interval 
$\{\pi \in \bl: \zhat \leq \pi \leq \Sigma\}$. Further the multiplicity of a root $\beta(S)$ only depends on the subgraph induced by $S$. This means for $\pi \in \bl[G(\Sigma)]$, $\mult \pi$ is the same whether we consider the 
 Kac-Moody algebra associated to $G$ or to $G(\Sigma)$. Thus Theorem \ref{pkost}{(b)} follows from \ref{pkost}{(a)}.

Next, equation \eqref{qkosdef} shows that $\kost[\,\beta(\sr)] = \sum_{\gamma \in \pp} \mult \gamma \;\, q^{|\gamma|}$. Corollary \ref{kostcor} now follows similarly, using the isomorphism of lemma \ref{posiso}. Corollary \ref{kostliecorr} follows from the easy observations that (i) the $q$-Kostant partition function reduces to the usual Kostant partition function at $q=1$, and (ii) the coefficient of $q^1$ in $K(\beta; \,q)$ is $\mult \beta$.

Finally, we prove proposition \ref{mobmult}. We only need to show that $(-1)^{l - |\pi|} \, \mult \pi$ satisfies the following defining relations of the M\"obius function: (i) $\mob(\zhat,\zhat)=1$, and 
(ii) $\sum_{\zhat \leq \pi \leq \Sigma} \mob(\zhat,\pi)=0$ for all $\Sigma \neq \zhat$. 
We have $\mult \,\zhat=1$ since each simple root is of multiplicity 1. Further, for $\Sigma =\{S_1, \cdots, S_k\} \in \bl$, we have $\sum_{\zhat \leq \pi \leq \Sigma} (-1)^{l - |\pi|} \, \mult \pi = \pgsig[1]{\Sigma}$ by Theorem \ref{pkost}(b). If $\Sigma \neq \zhat$, then there exists $j$ for which $S_j$ has two or more vertices. Thus $\pgsig[1]{S_j}=0$, and hence $\pgsig[1]{\Sigma}=\prod_{i=1}^k \pgsig[1]{S_i} =0$. \qed

\section{The Peterson recurrence formula} \label{three}
For $\beta \in Q_+$, set $c_{\beta}:=\sum\limits_{n\geq1}n^{-1} \,\mult{(\beta/n)}$. Then the Peterson recurrence formula \cite{kac} says:

\beq \label{prf1}
\form{\beta}{\beta-2\rho}\,c_{\beta}=\sum\limits_{\substack{(\beta',\beta'')\in Q_+ \times Q_+\\ \beta'+\beta''=\beta}}\form{\beta'}{\beta''} \,c_{\beta'}\,c_{\beta''}.
\eeq

Let $\bb :=\{ \beta(S): S \text{ is a connected subset of } \sr\}$. For $\beta = \beta(S) \in \bb$, we let $\supp \beta :=S$. It is easy to see that for $\beta \in \bb$, equation \eqref{prf1} becomes:

\begin{equation}\label{prf2}
 \form{\beta}{\beta-2\rho}\, \mult{\beta}=2\sum\limits_{\substack{\beta',\beta''\in \bb\\ \beta'+\beta''=\beta}}\form{\beta'}{\beta''}\,\mult{\beta'}\,\mult{\beta''}.
\end{equation}
where the factor of $2$ arises by taking unordered, rather than ordered, pairs $\beta',\beta''$ in the sum.
Now, $\form{\beta}{\beta-2\rho}=- 2 \, \ed(\supp \beta)$ and
$\form{\beta'}{\beta''}=- \ed(\supp\beta', \supp\beta'')$.

Suppose $\beta$ is not a simple root, i.e., $|\supp \beta| > 1$, then $\supp \beta$ has at least one edge, and 
equation \eqref{prf2} gives:

\begin{equation} \label{prf3}
 \mult {\beta}=\sum\limits_{\substack{\beta',\beta''\in \bb\\ \beta'+\beta''=\beta}}
\frac{\ed(\supp \beta', \supp \beta'')}{\ed(\supp \beta)} \, \mult {\beta'} \mult {\beta''}.
\end{equation}

Now, if $\gamma=\{\beta_1,\cdots,\beta_k\} \in \pp$, let $\gamma^\dag:=\{\beta\in \gamma: \beta \text{ is not a simple root}\}$.
Given $\gamma, \gamma' \in \pp$ with $\gamma \covers \gamma'$, we write  $\gamma = \{\beta_1,\cdots, \beta_k\}$ and  $\gamma' = \{\beta'_1,\cdots, \beta'_{k+1}\}$ such that $\beta_i = \beta'_i$ for $1 \leq i < k$ and $\beta_k = \beta'_k + \beta'_{k+1}$ (thus $\beta_k \in \gamma^\dag$). Define 
$$w(\gamma,\gamma'):= \frac{1}{|\gamma^\dag|}\frac{2\form{\beta_k'}{\beta_{k+1}'}}{\form{\beta_k}{\beta_k-2\rho}}
= \frac{1}{|\gamma^\dag|}\frac{\ed(\supp \beta_k',\supp \beta_{k+1}')}{\ed(\supp \beta_k)}.$$

We now have the following lemma.
\begin{lemma} \label{lemprf}
 For all $\hat{0}\neq \gamma\in \pp$, we have
\beq \label{prf4}
\mult \gamma=\sum\limits_{\substack{\gamma' \in \pp\\ \gamma \covers \gamma'}}w(\gamma,\gamma') \, \mult \gamma'.
\eeq
\end{lemma}

\begin{proof}
For $\gamma=\{\beta_1,\cdots,\beta_k\}$, equation \eqref{prf3} implies
$$\mult {\gamma}=\frac{1}{|\gamma^\dag|}\sum\limits_{\beta_i\in \gamma^\dag}\sum\limits_{\substack{\beta',\,\beta''\in \bb\\ \beta'+\beta''=\beta_i}}
\frac{\ed(\supp \beta', \supp \beta'')}{\ed(\supp \beta_i)} \, \mult {\beta'} \mult {\beta''} \prod_{j \neq i}\mult {\beta_j}.$$
The proof now follows from the fact that the elements $\gamma'$ covered by $\gamma$ are obtained by picking each $\beta_i \in \gamma^\dag$ and refining it into a sum of two elements of $\bb$ in all possible ways.\end{proof}

Since $\mult \zhat =1$, lemma \ref{lemprf} yields the following corollary on iteration.
\begin{corollary} \label{pcor}
Let $\gamma \in \pp$. Then
$$\mult \gamma = \sum_{\substack{p \text{ path in } \pp \\ \start(p) = \gamma \\ \eend(p) = \zhat}} w(p).$$
Here a {\em path in $\pp$} from $\gamma$ to $\zhat$ is a sequence $\gamma = \gamma_1 \covers \gamma_2 \cdots \covers \gamma_r = \zhat$, and its weight is defined to be $w(p) = \prod_{i=1}^{r-1} w(\gamma_i, \gamma_{i+1})$. The empty path ($\gamma = \zhat$) is taken to have weight 1.
\end{corollary}
Clearly, corollary \ref{pcor} is equivalent to proposition \ref{mpipath} via the isomorphism between $\pp$ and $\bl$ of lemma \ref{posiso}. As shown in the introduction, Theorem \ref{mainthm} now follows. \qed


\begin{thebibliography}{10}

\bibitem{erik2}
Henrik Eriksson and Kimmo Eriksson.
\newblock Conjugacy of {C}oxeter elements.
\newblock {\em Electron. J. Combin.}, 16(2, Special volume in honor of Anders
  Bjorner):Research Paper 4, 7, 2009.

\bibitem{humphreys}
James~E. Humphreys.
\newblock {\em Reflection groups and {C}oxeter groups}, volume~29 of {\em
  Cambridge Studies in Advanced Mathematics}.
\newblock Cambridge University Press, Cambridge, 1990.

\bibitem{kac}
V.~G. Kac.
\newblock {\em Infinite dimensional {L}ie algebras}.
\newblock Cambridge University Press, third edition, 1990.

\bibitem{mac-mor}
Matthew Macauley and Henning~S. Mortveit.
\newblock On enumeration of conjugacy classes of {C}oxeter elements.
\newblock {\em Proc. Amer. Math. Soc.}, 136(12):4157--4165, 2008.

\bibitem{rcread}
Ronald~C. Read.
\newblock An introduction to chromatic polynomials.
\newblock {\em J. Combinatorial Theory}, 4:52--71, 1968.

\bibitem{rota}
Gian-Carlo Rota.
\newblock On the foundations of combinatorial theory. {I}. {T}heory of
  {M}\"obius functions.
\newblock {\em Z. Wahrscheinlichkeitstheorie und Verw. Gebiete}, 2:340--368
  (1964), 1964.

\bibitem{jyshi1}
Jian-Yi Shi.
\newblock The enumeration of {C}oxeter elements.
\newblock {\em J. Algebraic Combin.}, 6(2):161--171, 1997.

\bibitem{jyshi2}
Jian-yi Shi.
\newblock Conjugacy relation on {C}oxeter elements.
\newblock {\em Adv. Math.}, 161(1):1--19, 2001.

\bibitem{stanley-acyc}
Richard~P. Stanley.
\newblock Acyclic orientations of graphs.
\newblock {\em Discrete Math.}, 5:171--178, 1973.

\bibitem{stanley-sfgcp}
Richard~P. Stanley.
\newblock A symmetric function generalization of the chromatic polynomial of a
  graph.
\newblock {\em Adv. Math.}, 111(1):166--194, 1995.

\bibitem{rvsv}
R.~Venkatesh and Sankaran Viswanath.
\newblock Unique factorization of tensor products for {K}ac-{M}oody algebras.
\newblock {\em Adv. Math.}, 231(6):3162--3171, 2012.

\end{thebibliography}

\end{document}